\documentclass[reqno]{amsart}
\usepackage{amssymb,latexsym,amsmath,amsthm,enumerate,amsbsy}
\usepackage[mathscr]{eucal}
\usepackage{framed,color,graphicx}
\usepackage{mathrsfs}
\usepackage[all]{xy}
\usepackage{tikz}
\usepackage{cite}
\usepackage{subcaption}

\makeatletter
\@namedef{subjclassname@2020}{\textup{2020} Mathematics Subject Classification}
\makeatother

\usetikzlibrary{positioning,decorations.pathreplacing,patterns,decorations.pathmorphing}
\tikzset{%
element/.style={draw, shape=circle, fill=white, inner sep=1.4pt}
}

\DeclareSymbolFont{bbold}{U}{bbold}{m}{n}
\DeclareSymbolFontAlphabet{\mathbbold}{bbold}

\theoremstyle{plain}
\newtheorem{thm}{Theorem}[section]
\newtheorem{lem}[thm]{Lemma}
\newtheorem{cor}[thm]{Corollary}
\newtheorem{pro}[thm]{Proposition}

\newtheorem{problem}[thm]{Problem}

\newtheorem{claim}{Claim}[section]

\theoremstyle{definition}

\newcommand{\bp}{\mathbf{p}}
\newcommand{\bq}{\mathbf{q}}
\newcommand{\br}{\mathbf{r}}
\newcommand{\bs}{\mathbf{s}}
\newcommand{\bt}{\mathbf{t}}
\newcommand{\bu}{\mathbf{u}}

\newcommand{\bw}{\mathbf{w}}

\begin{document}

\title[A nonfinitely based ai-semiring of order four]
{A nonfinitely based additively idempotent semiring of order four}

\author{Mengya Yue}
\address{School of Mathematics, Northwest University, Xi'an, 710127, Shaanxi, P.R. China}
\email{myayue@yeah.net}

\author{Miaomiao Ren}
\address{School of Mathematics, Northwest University, Xi'an, 710127, Shaanxi, P.R. China}
\email{miaomiaoren@yeah.net}


\subjclass[2020]{16Y60, 03C05, 08B05}
\keywords{semiring, variety, nonfinitely based}

\begin{abstract}
We first establish a sufficient condition for an additively idempotent semiring to be nonfinitely based.
As applications, we exhibit several examples of additively idempotent semirings satisfying this condition,
including a $4$-element semiring $S_{(4,124)}$ whose additive reduct has two minimal elements and two coatoms.
Consequently, these semirings have no finite basis for their identities.
\end{abstract}

\maketitle

\section{Introduction}
By an \emph{additively idempotent semiring} (or ai-semiring for short) we mean
an algebra $(S, +, \cdot)$ such that
\begin{itemize}
\item the additive reduct $(S, +)$ is a commutative idempotent semigroup;

\item the multiplicative reduct $(S, \cdot)$ is a semigroup;

\item the distributive laws hold:
\[
x(y+z)\approx xy+xz, \quad (x+y)z\approx xz+yz.
\]
\end{itemize}
A \emph{commutative ai-semiring} is an ai-semiring whose multiplicative reduct is commutative.

Additively idempotent semirings arise naturally in several areas where classical algebraic methods encounter limitations.
They form the algebraic foundation of tropical geometry~\cite{ms} and provide a natural setting for the idempotent superposition principle, which linearizes certain nonlinear problems in optimization, idempotent linear algebra, and interval analysis~\cite{litvinov2011universal}.
Moreover, they have found applications in theoretical computer science, for instance in modal and game semirings for program verification~\cite{moller2013}.

On an ai-semiring $S$, one defines a binary relation $\leq$ by
\[
a \leq b \Leftrightarrow a + b = b.
\]
This relation is a partial order and is compatible with both addition and multiplication;
hence ai-semirings are sometimes called \emph{semilattice-ordered semigroups} \cite{kp}.
Unless stated otherwise, any order-theoretic statement concerning an ai-semiring refers to this order.

A \emph{variety} of ai-semirings is a class of ai-semirings
closed under taking subalgebras, homomorphic images, and arbitrary direct products.
Birkhoff's celebrated theorem states that a class of ai-semirings is a variety if and only if it is an equational class;
that is, the class of all ai-semirings that satisfy a certain set of identities.
A variety of ai-semirings is \emph{finitely based} if it admits a finite equational basis
(i.e., it can be defined by a finite set of identities);
otherwise, it is \emph{nonfinitely based}.
An ai-semiring $S$ is finitely based (resp., nonfinitely based)
if the variety $\mathsf{V}(S)$ it generates is finitely based (resp., nonfinitely based).

The \emph{finite basis problem} for a class of ai-semirings concerns the classification of
its members according to whether they are finitely based.
Over the past two decades, this problem for ai-semirings of small order has been intensively studied,
and considerable progress has been made
(see~\cite{do, dolinka, dgv25, gmrz, gpz, gv2501, jrz, pas05, rjzl, rlzc, rlyc, shap23, sr, Volkov, wrz, yrg, yrzs, zrc}).

In particular, Dolinka~\cite{do} found the first example of a nonfinitely based ai-semiring,
which has exactly $7$ elements.
Volkov~\cite{Volkov} proved that the ai-semiring whose multiplicative reduct is a
$6$-element Brandt monoid has no finite equational basis.
Subsequently,
Gusev and Volkov~\cite{gv2501} provided another $6$-element nonfinitely based ai-semiring.

Shao and Ren~\cite{sr} demonstrated that every variety generated by ai-semirings of order two is finitely based.
Up to isomorphism, there are exactly $61$ ai-semirings of order three, denoted by $S_{i}$ for $1 \leq i \leq 61$.
Zhao et al.~\cite{zrc} established that all of them are finitely based,
with the possible exception of $S_7$ (whose Cayley tables are provided in Table~\ref{tb24111401}).
Jackson et al.~\cite{jrz} later confirmed that $S_7$ itself is nonfinitely based.

\begin{table}[ht]
\caption{The Cayley tables of $S_7$} \label{tb24111401}
\begin{tabular}{c|ccc}
$+$      &$0$&$a$&$1$\\
\hline
$0$      &$0$&$0$&$0$\\
$a$      &$0$&$a$&$0$\\
$1$      &$0$&$0$&$1$\\
\end{tabular}\qquad
\begin{tabular}{c|ccc}
$\cdot$      &$0$&$a$&$1$\\
\hline
$0$      &$0$&$0$&$0$\\
$a$      &$0$&$0$&$a$\\
$1$      &$0$&$a$&$1$\\
\end{tabular}
\end{table}

Moreover, Jackson et al.~\cite{jrz}, Gusev and Volkov~\cite{gv2301, gv2302, gv2501}, and Gao et al.~\cite{gmrz, gjrz2}
showed that the nonfinite basis property of $S_7$ can be transferred to many finite ai-semirings whose varieties contain it.
This observation led Gao et al.~\cite{gjrz2} to propose the following problem:
\begin{problem}\label{prob0127}
Is every finite ai-semiring whose variety contains $S_7$ nonfinitely based$?$
\end{problem}

Up to isomorphism, there are exactly $866$ ai-semirings of order four, denoted by $S_{(4, i)}$ for $1 \leq i \leq 866$.
These algebras are classified into five distinct types based on their additive orders, as illustrated in \cite[Figure 1]{rlzc}.
The finite basis problem for the first three types has been resolved by Gao et al.~\cite{gmrz}, Ren et al.~\cite{rjzl, rlyc, rlzc}, Shaprynski\v{\i}~\cite{shap23}, Wu et al.~\cite{wrz}, and Yue et al.~\cite{yrzs}.
The finite basis problem for the fourth type has been settled.
The results will be presented in a comprehensive paper by Ren et al.~\cite{ryy}, which systematically summarizes the entire type.
Additionally, Yue et al.~\cite{yrg} solved two algebras within this type;
the varieties generated by these two algebras both contain $S_7$, and they were shown to be nonfinitely based,
thereby providing further evidence for Problem~\ref{prob0127}.

Recently, we have begun investigating the last remaining type of four-element ai-semirings,
namely the class $\{S_{(4,i)} \mid 59 \leq i \leq 275\}$, which consists of $217$ algebras
whose additive reducts have two minimal elements and two coatoms.
Up to now, four of these $217$ semirings remain unresolved.
Among these four, only $S_{(4,124)}$ (whose Cayley tables are presented in Tables~\ref{tb124})
generates a variety containing $S_7$, making it an interesting and worthy candidate for independent study.

The present paper focuses on the finite basis problem for $S_{(4, 124)}$.
Observe that the additive reduct of $S_{(4, 124)}$ has two minimal elements $3$, $4$ and two coatoms $2$, $4$.
We shall demonstrate that the variety $\mathsf{V}(S_{(4,124)})$ also contains $S_7$,
together with $S_2$ and $S_{53}$ (see Tables~\ref{tbS2} and~\ref{tb53} for the Cayley tables of the latter two).
Even though $S_{(4,124)}$ is only a $4$-element commutative ai-semiring and thus appears deceptively simple,
none of the known sufficient conditions for nonfinite basis in the literature apply to it.

We first develop a sufficient condition for an ai-semiring to be nonfinitely based,
and then apply it to solve the finite basis problem for $S_{(4, 124)}$,
which constitutes a contribution to Problem~\ref{prob0127}.

\begin{table}[ht]
\caption{Cayley tables of $S_{(4,124)}$, $S_2$, and $S_{53}$}
\centering
\subcaptionbox{$S_{(4,124)}$\label{tb124}}{
\begin{tabular}{c|cccc}
$+$ & $1$ & $2$ & $3$ & $4$ \\
\hline
$1$ & $1$ & $1$ & $1$ & $1$ \\
$2$ & $1$ & $2$ & $2$ & $1$ \\
$3$ & $1$ & $2$ & $3$ & $1$ \\
$4$ & $1$ & $1$ & $1$ & $4$ \\
\end{tabular}\qquad
\begin{tabular}{c|cccc}
$\cdot$ & $1$ & $2$ & $3$ & $4$ \\
\hline
$1$ & $1$ & $1$ & $1$ & $1$ \\
$2$ & $1$ & $1$ & $2$ & $1$ \\
$3$ & $1$ & $2$ & $3$ & $4$ \\
$4$ & $1$ & $1$ & $4$ & $2$ \\
\end{tabular}}\\
\subcaptionbox{$S_2$\label{tbS2}}{
\begin{tabular}{c|ccc}
$+$ & $1$ & $2$ & $3$ \\
\hline
$1$ & $1$ & $1$ & $1$ \\
$2$ & $1$ & $2$ & $1$ \\
$3$ & $1$ & $1$ & $3$ \\
\end{tabular}\qquad
\begin{tabular}{c|ccc}
$\cdot$ & $1$ & $2$ & $3$ \\
\hline
$1$ & $1$ & $1$ & $1$ \\
$2$ & $1$ & $1$ & $1$ \\
$3$ & $1$ & $1$ & $2$ \\
\end{tabular}}
\qquad
\subcaptionbox{$S_{53}$\label{tb53}}{
\begin{tabular}{c|ccc}
$+$ & $1$ & $2$ & $3$ \\
\hline
$1$ & $1$ & $1$ & $3$ \\
$2$ & $1$ & $2$ & $3$ \\
$3$ & $3$ & $3$ & $3$ \\
\end{tabular}\qquad
\begin{tabular}{c|ccc}
$\cdot$ & $1$ & $2$ & $3$ \\
\hline
$1$ & $3$ & $1$ & $3$ \\
$2$ & $1$ & $2$ & $3$ \\
$3$ & $3$ & $3$ & $3$ \\
\end{tabular}}
\end{table}

\section{Preliminaries}
In this section, we introduce the basic terminology, notation, and preliminary results that will be used throughout the paper.
All definitions and results are stated for general (not necessarily commutative) ai-semirings.
When restricted to the commutative case, some of these notions simplify; we will not repeat these routine adjustments.

Let $X$ be a countably infinite set of variables, and let $X^+$ denote the free semigroup over $X$.
An \textit{ai-semiring term} (or simply a \textit{term}) over $X$ is defined as a finite nonempty set of words in $X^+$.
(In the sequel, terms are denoted by bold lowercase letters $\mathbf{u}, \mathbf{v}, \mathbf{w}, \dots$,
while ordinary lowercase letters $x, y, z, \dots$ represent variables.)
A term is represented as a formal sum of its elements.
Specifically, $\mathbf{w} = \mathbf{u}_1 + \mathbf{u}_2 + \cdots + \mathbf{u}_n$
indicates that $\mathbf{w} = \{ \mathbf{u}_1, \mathbf{u}_2, \dots, \mathbf{u}_n \}$.
The order of the summands in the formal sum is irrelevant,
and multiple occurrences of the same word are collapsed into a single occurrence.
Two terms are equal if and only if their underlying sets coincide.

Let $P_f(X^+)$ denote the collection of all terms over $X$.
This set forms an ai-semiring under the usual operations of term addition and multiplication.
It follows from ~\cite[Theorem 2.5]{kp} that
$P_f(X^+)$ is free over $X$ in the variety of all ai-semirings.
An \textit{ai-semiring substitution} (or just a \textit{substitution})
is defined as a semiring homomorphism from $P_f(X^+)$ to itself.

Let $\mathbf{u}$ and $\mathbf{v}$ be terms.
We say that $\mathbf{u}$ is a \emph{subterm of $\mathbf{v}$} if there exist terms $\mathbf{p}_1$ and $\mathbf{p}_2$ such that
\[
\mathbf{v}=\mathbf{p}_1\mathbf{u}\mathbf{p}_2+\mathbf{p}_3,
\]
where $\mathbf{p}_1$ and $\mathbf{p}_2$ may be the empty word, and $\mathbf{p}_3$ may be the empty set.

An \emph{ai-semiring identity} (or simply an \emph{identity})
is a formal expression of the form $\mathbf{u} \approx \mathbf{v}$, where $\mathbf{u}$ and $\mathbf{v}$ are terms.
Let $S$ be an ai-semiring and $\mathbf{u} \approx \mathbf{v}$ an identity.
We say that $S$ \emph{satisfies $\mathbf{u} \approx \mathbf{v}$}, or that $\mathbf{u} \approx \mathbf{v}$ \emph{holds in $S$},
if $\varphi(\mathbf{u}) = \varphi(\mathbf{v})$ for every semiring homomorphism $\varphi: P_f(X^+) \to S$.

The following result, which is a consequence of \cite[Lemma 3.1]{dolinka},
concerns the equational logic of ai-semirings.

\begin{lem}\label{nlemma1}
Let $\Sigma$ be a set of identities and let $\mathbf{u} \approx \mathbf{v}$ be a nontrivial identity.
Then $\mathbf{u} \approx \mathbf{v}$ is derivable from $\Sigma$ if and only if there exist terms
$\mathbf{t}_1, \mathbf{t}_2, \ldots, \mathbf{t}_n \in P_f(X^+)$ such that $\mathbf{u} = \mathbf{t}_1$,
$\mathbf{v} = \mathbf{t}_n$ and, for each $1 \leq i < n$, there are terms
$\mathbf{p}_i, \mathbf{q}_i, \mathbf{r}_i, \mathbf{s}_i, \mathbf{s}'_i \in P_f(X^+)$ and a substitution
$\varphi_i\colon P_f(X^+) \to P_f(X^+)$ such that
\[
        \mathbf{t}_i = \mathbf{p}_i\varphi_i(\mathbf{s}_i)\mathbf{q}_i  + \mathbf{r}_i, \quad
        \mathbf{t}_{i+1} = \mathbf{p}_i\varphi_i(\mathbf{s}'_i)\mathbf{q}_i + \mathbf{r}_i,
\]
where $\mathbf{s}_i \approx \mathbf{s}'_i \in \Sigma$ or $\mathbf{s}'_i \approx \mathbf{s}_i \in \Sigma$,
$\mathbf{p}_i$ and $\mathbf{q}_i$ may be the empty word,
and $\mathbf{r}_i$ may be the empty set.
\end{lem}

We denote by $\mathbf{u} \preceq \mathbf{v}$ (or equivalently $\mathbf{v} \succeq \mathbf{u}$)
the identity $\mathbf{v}\approx \mathbf{v} + \mathbf{u}$,
which we refer to as an \textit{ai-semiring inequality} (or simply an \textit{inequality}).
It is routine to verify that an ai-semiring $S$ satisfies the inequality $\mathbf{u} \preceq \mathbf{v}$
if and only if for every semiring homomorphism $\varphi: P_f(X^+) \to S$,
we have $\varphi(\mathbf{u}) \leq \varphi(\mathbf{v})$.
Consequently, $S$ satisfies an identity $\mathbf{u} \approx \mathbf{v}$ precisely when it satisfies
both inequalities $\mathbf{u} \preceq \mathbf{v}$ and $\mathbf{v} \preceq \mathbf{u}$.
Therefore, $\mathbf{u} \approx \mathbf{v}$ is equivalent to
the inequalities $\mathbf{u} \preceq \mathbf{v}$ and $\mathbf{v} \preceq \mathbf{u}$.
For this reason, when dealing with a set of identities,
we may always assume without loss of generality that it consists of inequalities.
The following discussion shows that it suffices to consider inequalities of a special form.

Now let $\Sigma$ be a set of identities, and let $\mathbf{u} \approx \mathbf{v}$ be an identity such that
\[
\mathbf{u} = \mathbf{u}_1 + \cdots + \mathbf{u}_k, \quad
\mathbf{v} = \mathbf{v}_1 + \cdots + \mathbf{v}_\ell,
\]
where $\mathbf{u}_i, \mathbf{v}_j \in X^+$ for $1 \leq i \leq k$ and $1 \leq j \leq \ell$.
One readily checks that the ai-semiring variety defined by $\mathbf{u} \approx \mathbf{v}$
coincides with the ai-semiring variety defined by the inequalities
\[
\mathbf{u}_i \preceq \mathbf{v}, \quad \mathbf{v}_j \preceq \mathbf{u} \quad (1 \leq i \leq k, \, 1 \leq j \leq \ell).
\]
Consequently, to prove that $\mathbf{u} \approx \mathbf{v}$ is derivable from $\Sigma$,
it suffices to show that for every $i$ and $j$,
the inequalities $\mathbf{u}_i \preceq \mathbf{v}, \, \mathbf{v}_j \preceq \mathbf{u}$ can be derived from $\Sigma$.
In view of this, we always restrict our attention to the inequalities of the form
$\mathbf{q} \preceq \mathbf{u}$, where $\mathbf{q}$ is a word and $\mathbf{u}$ is a term.

To end this section, we introduce some notation.
Let $\bw$ be a nonempty word in $X^+$, and let $x$ be a variable in $X$. Then
\begin{itemize}
\item $c(\bw)$ denotes the set of all variables occurring in $\bw$;

\item $\ell(\bw)$ denotes the number of variables occurring in $\bw$ counting multiplicities;

\item $S_2(\bw)$ denotes the set of all subwords $\bp$ of $\bw$ with $\ell(\bp)=2$;

\item $occ(x, \bw)$ denotes the number of occurrences of $x$ in $\bw$.
\end{itemize}

Now let $\bu$ be a term such that $\bu=\bu_1+\bu_2+\cdots+\bu_n$,
where $\bu_i \in X^+$, $1 \leq i \leq n$.
Let $k$ be a positive integer. Then
\begin{itemize}
\item $c(\bu)$ denotes the set $\bigcup_{i=1}^n c(\bu_i)$;

\item $L_{ k}(\bu)$ denotes the set $\{\bu_i \in \bu \mid \ell(\bu_i)= k\}$;

\item $S_2(\bu)$ denotes the set $\bigcup_{i=1}^n S_2(\bu_i)$;

\item $\delta(\bu)$ denotes the set of subsets $Z$ of $c(\bu)$ such that for every
 $\bu_i\in\bu$, $Z\cap c(\bu_i)$ is a singleton and $occ(x,\bu_i)=1$ if $\{x\}=Z\cap c(\bu_i)$.
\end{itemize}

\section{The semiring $S_{(4, 124)}$}
In this section, we use a syntactic approach to show that the ai-semiring $S_{(4, 124)}$ is nonfinitely based.

The following three lemmas provide solutions to the equational problems for $S_2$, $S_7$, and $S_{53}$, respectively.
They are due to Ren et al.~\cite[Lemma 4.1]{rlzc}, Jackson et al.~\cite[Proposition 5.5]{jrz}, and Yue et al.~\cite[Lemma 3.6]{yrg}.

\begin{lem}\label{lem201}
Let $\bq\preceq \bu$ be an inequality such that
$\bu=\bu_1+\bu_2+\cdots+\bu_n$ with $\bu_i, \bq \in X^+$ for $1\leq i \leq n$.
Then $\bq\preceq \bu$ holds in $S_2$ if and only if one of the following conditions is satisfied:
\begin{itemize}
\item[$(1)$] $\ell(\bu_i)\geq 3$ for some $\bu_i \in \bu$;

\item[$(2)$] $c(L_1(\bu))\cap c(L_2(\bu)) \neq \emptyset$;

\item[$(3)$] $\ell(\bu_i)\leq 2$ for all $\bu_i \in \bu$, and $c(L_1(\bu))\cap c(L_2(\bu))=\emptyset$.
In this case, $\ell(\mathbf{q}) \leq 2$.
Moreover, if $\ell(\mathbf{q}) = 1$, then $\mathbf{q} \preceq \mathbf{u}$ is trivial;
if $\ell(\mathbf{q}) = 2$, then $c(\mathbf{q}) \subseteq c(L_2(\mathbf{u}))$.
\end{itemize}
\end{lem}

\begin{lem}\label{lemma21}
Let $\bq\preceq \bu$ be an inequality such that
$\bu=\bu_1+\bu_2+\cdots+\bu_n$ with $\bu_i, \bq \in X^+$ for $1\leq i \leq n$.
Then $\bq\preceq\bu$ holds in $S_7$ if and only if $c(\bq)\subseteq c(\bu)$ and $\delta(\bu)\subseteq\delta(\bu+\bq)$.
\end{lem}

\begin{lem}\label{lem5301}
Let $\bq\preceq \bu$ be an inequality such that
$\bu=\bu_1+\bu_2+\cdots+\bu_n$ with $\bu_i, \bq \in X^+$ for $1\leq i \leq n$.
Then $\bq\preceq\bu$ holds in $S_{53}$
if and only if $c(\bq)\subseteq c(\bu)$, and one of the following conditions is satisfied:
\begin{itemize}
\item[$(1)$] $L_{\geq 2}(\bu)= \emptyset$. In this case, $\mathbf{q} \preceq \mathbf{u}$ is trivial;

\item[$(2)$] $L_{\geq 2}(\bu)\neq \emptyset$,
and for every $\bw\in S_2(\bq)$ there exists $\bw'\in S_2(\bu)$ such that $c(\bw')\subseteq c(\bw)$.
\end{itemize}
\end{lem}

Recall that the \emph{distance} $d(x,y)$ between two vertices $x$ and $y$ in
a connected graph is the length of a shortest path connecting them.
A graph is \emph{bipartite} if its vertex set can be partitioned into two disjoint
sets such that no two vertices within the same set are adjacent.
By \cite[Theorem 4]{bol98}, a graph is bipartite if and only if it contains no odd cycle.
The following lemma is essentially a direct consequence of \cite[Theorem 4]{bol98};
we include a short proof for completeness.

\begin{lem}\label{graph01}
Let $G$ be a bipartite graph,
and let $H$ be a set of vertices of $G$ such that no two distinct vertices in $H$ are joined by an odd-length path in $G$.
Then $G$ has a bipartition $(Y, Z)$ with $H \subseteq Y$.
\end{lem}
\begin{proof}
Let $\{C_i\mid i \in I\}$ be the set of all connected components of $G$,
and for each $i \in I$, denote by $V(C_i)$ the vertex set of $C_i$.
We choose a vertex $v_i \in V(C_i)$ as follows:
if $H \cap V(C_i) \neq \emptyset$, then $v_i \in H \cap V(C_i)$;
otherwise, $v_i$ is an arbitrary vertex in $V(C_i)$.
Now define
\[
Y_i = \{ x \in V(C_i) \mid d(x, v_i) \text{ is even} \}, \quad
Z_i = V(C_i) \setminus Y_i.
\]

If an edge connected two vertices both in $Y_i$ or both in $Z_i$, then taking shortest paths from $v_i$ to those vertices together with that edge would yield an odd cycle (since the sum of two even distances or two odd distances is even, and adding the edge contributes one more). This contradicts the fact that $G$ contains no odd cycle. Hence no edge joins two vertices within $Y_i$ or within $Z_i$, so $(Y_i, Z_i)$ is a bipartition of $C_i$.

Let
\[
Y = \bigcup_{i \in I} Y_i, \quad Z = \bigcup_{i \in I} Z_i.
\]
Then $G$ is a bipartite graph with bipartition $(Y, Z)$.

Now take any $h \in H$. There exists a unique $i \in I$ such that $h \in V(C_i)$.
By our choice of $v_i$, the distance $d(h, v_i)$ is even (otherwise an odd-length path would connect two vertices of $H$, contradicting the hypothesis). Hence $h \in Y_i \subseteq Y$. Therefore $H \subseteq Y$, completing the proof.
\end{proof}

For an integer $n \geq 1$, we define the terms
\[
\mathbf{u}^{(n)} = x_1x_2 + x_2x_3 + \cdots + x_{2n}x_{2n+1} + x_{2n+1}x_1 + y_1y_2 + y_2y_1 + y_1,
\]
and
\[
\mathbf{q}^{(n)} = y_2.
\]
Let $\mathcal{W}$ denote the ai-semiring variety defined by the inequalities $\mathbf{q}^{(n)} \preceq \mathbf{u}^{(n)}$ for all $n \geq 1$.
Inequalities analogous to $\mathbf{q}^{(n)} \preceq \mathbf{u}^{(n)}$ have proven to be very useful;
they have played a key role in establishing the nonfinite basis property for several algebras (see, e.g.,~\cite{aj25, gmrz, jrz, wrz, yrg}).

We now verify that $S_2$, $S_7$, and $S_{53}$ belong to $\mathcal{W}$.
In each case, the verification follows directly from the description of the inequalities and the relevant lemmas.
For $S_2$, observe that $c(L_1(\mathbf{u}^{(n)})) \cap c(L_2(\mathbf{u}^{(n)})) \neq \emptyset$,
so Lemma~\ref{lem201} applies and gives $\mathbf{q}^{(n)} \preceq \mathbf{u}^{(n)}$ in $S_2$ for every $n$.
For $S_7$, we have that $c(\mathbf{q}^{(n)}) \subseteq c(\mathbf{u}^{(n)})$ and $\delta(\mathbf{u}^{(n)}) = \emptyset$,
hence Lemma~\ref{lemma21} yields the desired inequality in $S_7$.
For $S_{53}$, the conditions $c(\mathbf{q}^{(n)}) \subseteq c(\mathbf{u}^{(n)})$, $L_2(\mathbf{u}^{(n)}) \neq \emptyset$,
and $\mathbf{q}^{(n)}$ being a single variable together imply, by Lemma~\ref{lem5301},
that $\mathbf{q}^{(n)} \preceq \mathbf{u}^{(n)}$ holds in $S_{53}$.

Consequently, $\mathcal{W}$ contains $S_2$, $S_7$, and $S_{53}$.
With this verification in place, we now state and prove the main result of this section.

\begin{thm}\label{thm1}
Let $\mathcal{V}$ be an ai-semiring variety.
If $\mathcal{V}$ is a subvariety of the variety $\mathcal{W}$ and
contains $S_2$, $S_7$, and $S_{53}$, then $\mathcal{V}$ is nonfinitely based.
\end{thm}

\begin{proof}
Let $\mathcal{V}$ be a subvariety of $\mathcal{W}$ that contains $S_2$, $S_7$, and $S_{53}$,
and let $\mathsf{Id}(\mathcal{V})$ denote the set of all identities
over $X$ that hold in $\mathcal{V}$.
Then $\mathcal{V}$ satisfies the inequality $\bq^{(n)} \preceq \bu^{(n)}$ for all $n \geq 1$.
The strategy of the proof is as follows.
We shall show that for every $n \geq 1$, the set $\Sigma_n$
of all $n$-variable identities in $\mathsf{Id}(\mathcal{V})$ does not form a basis for $\mathsf{Id}(\mathcal{V})$.
To this end, it suffices to prove that for every $n \geq 1$,
the inequality $\bq^{(n)} \preceq \bu^{(n)}$ can not be derived by $\Sigma_n$.
In fact, we shall show a stronger result:
$\bq^{(n)} \preceq \bu^{(n)}$ is not derivable from $\Sigma_n\cup \{xy\approx yx\}$.
Since commutativity is assumed, we may work entirely within the framework of commutative ai-semiring terms;
that is, the formal sum of words in the free commutative semigroup over $X$.

Suppose for contradiction that $\bq^{(n)} \preceq \bu^{(n)}$ is derivable from $\Sigma_n\cup \{xy\approx yx\}$.
Then by Lemma \ref{nlemma1} there exists a nontrivial identity $\bs_1 \approx \bs'_1$ in $\Sigma_n$
and a substitution $\varphi$ such that $\varphi(\bs_1) \neq \varphi(\bs'_1)$,
and $\varphi(\bs_1)$ is a subterm of $\bu^{(n)}$.
Then
\begin{equation}\label{26042301}
\bp_1\varphi(\bs_1)+\br_1 = \bu^{(n)}
\end{equation}
for some terms $\bp_1$ and $\br_1$,
where $\bp_1$ may be the set $\{1\}$, and $\br_1$ may be the empty set.
From this, we will ultimately derive that $\varphi(\bs_1) = \varphi(\bs'_1)$,
contradicting the earlier conclusion $\varphi(\bs_1) \neq \varphi(\bs'_1)$.

Firstly, it is easy to see that $\bs_1$ satisfies the following conditions:
\begin{itemize}
\item[$(a)$] $\ell(\bw)\leq 2$ for all $\bw \in \bs_1$;

\item[$(b)$] $\bw$ is linear for all $\bw \in \bs_1$;

\item[$(c)$] In the subset representation, $\bs_1$ can not contain
\[
\{x_1x_2, x_2x_3, \ldots, x_{2m}x_{2m+1}, x_{2m+1}x_1\}
\]
for any $m\geq 1$.
\end{itemize}

By condition $(a)$ we have that $\bs_1$ is the disjoint union of $L_1(\bs_1)$ and $L_2(\bs_1)$.
Since $\mathcal{V}$ satisfies $\bs_1 \approx \bs'_1$ and contains $S_{53}$,
it follows that $S_{53}$ also satisfies $\bs_1 \approx \bs'_1$.
Applying Lemma~\ref{lem5301}, one sees that $L_2(\mathbf{s}_1)$ cannot be empty;
otherwise $\mathbf{s}_1 \approx \mathbf{s}_1'$ would be trivial,
contradicting the fact that it is a nontrivial identity.
Combining $L_2(\mathbf{s}_1)\neq \emptyset$ with \eqref{26042301},
we obtain that $\bp_1$ must be the empty word,
and so
\begin{equation}\label{26042302}
\varphi(\bs_1)+\br_1 = \bu^{(n)}.
\end{equation}
Now we may regard $L_2(\bs_1)$ as a graph with vertex set $c(L_2(\bs_1))$
and edge set $\big\{ \{x, y\} \mid xy \in L_2(\bs_1) \big\}$.

\begin{claim}\label{claim02}
If $x, y\in c(L_1(\bs_1)) \cap c(L_2(\bs_1))$, then there is no path of odd length connecting $x$ and $y$.
\end{claim}

\begin{proof}[Proof of Claim $\ref{claim02}$]
Suppose, to the contrary, that such a path exists. Then there are variables $x_1, x_2, \dots, x_{2k} \in X$ ($k \geq 0$) such that
\[
x,\; xx_1,\; x_1x_2,\; \dots,\; x_{2k}y,\; y
\]
all lie in $\mathbf{s}_1$. The case $k = 0$ corresponds to the direct edge $xy \in \mathbf{s}_1$.
From \eqref{26042302}, we know that $\varphi(z) = y_1$ for every variable $z \in c(L_1(\mathbf{s}_1)) \cap c(L_2(\mathbf{s}_1))$. Since $x$ is such a variable, we have $\varphi(x) = y_1$.

Now traverse the path step by step.
Starting from $\varphi(x) = y_1$, using \eqref{26042302} we obtain $\varphi(x_1) = y_2$,
then $\varphi(x_2) = y_1$, and so on.
Repeating this process, after an odd number of steps we conclude that $\varphi(y) = y_2$.

On the other hand, since $y$ itself is also in $c(L_1(\mathbf{s}_1)) \cap c(L_2(\mathbf{s}_1))$,
it follows from \eqref{26042302} that $\varphi(y) = y_1$.
Thus $\varphi(y) = y_1 = y_2$, which is impossible.
This contradiction shows that no such odd path can exist, thereby establishing Claim~\ref{claim02}.
\end{proof}

\begin{claim}\label{claim01}
$L_2(\bs'_1) \subseteq \bs_1$.
\end{claim}

\begin{proof}[Proof of Claim $\ref{claim01}$]
Suppose that $\mathbf{p}$ is a word in $\bs'_1$ that does not belong to $\mathbf{s}_1$.
Then the nontrivial inequality $\mathbf{p} \preceq \mathbf{s}_1$ holds in $\mathcal{V}$ and also holds in $S_{53}$.
By Lemma~\ref{lem5301} and condition (a),
there exists $\mathbf{t} \in L_2(\mathbf{s}_1)$ such that $c(\mathbf{t}) \subseteq c(\mathbf{p})$.
Together with condition (b),
this yields $\mathbf{p} = \mathbf{t}$, and hence $\mathbf{p} \in \mathbf{s}_1$, a contradiction.
This completes the proof.
\end{proof}

With the above preparations in place, we now complete the proof of the theorem by distinguishing two cases.

\textbf{Case 1.} $c(L_1(\bs_1)) \cap c(L_2(\bs_1))=\emptyset$.

Let $\bp$ be an arbitrary word in $\bs'_1$.
Then the inequality $\bp \preceq \bs_1$ is satisfied by $\mathcal{V}$, and therefore also holds in $S_2$.
By Lemma~\ref{lem201} and condition (a), we have $\ell(\bp) \leq 2$.
If $\ell(\bp)=1$, then by Lemma~\ref{lem201}, the inequality $\bp \preceq \bs_1$ is trivial, and hence $\bp \in \bs_1$.
If $\ell(\bp)=2$, then by Claim~\ref{claim01}, $\bp \in \bs_1$.
Therefore, $\bs'_1 \subseteq \bs_1$.

We now use $\bs'_1 \subseteq \bs_1$ to prove the reverse containment $\bs_1 \subseteq \bs'_1$.
Let $\bq$ be an arbitrary word in $\bs_1$.
Then the inequality $\bq \preceq \bs'_1$ holds in $\mathcal{V}$, and hence is satisfied by $S_2$.
Since $\bs'_1 \subseteq \bs_1$, it follows from conditions $(a)$ and $(b)$ that
every word $\bw$ in $\bs'_1$ is linear with $\ell(\bw) \leq 2$,
and $c(L_1(\bs'_1)) \cap c(L_2(\bs'_1)) = \emptyset$.
One can also show that $L_2(\mathbf{s}_1) \subseteq \mathbf{s}'_1$; the proof is similar to that of Claim~\ref{claim01}.
The remainder of the proof follows the same line of reasoning as the proof of $\mathbf{s}_1' \subseteq \mathbf{s}_1$;
we omit the details.
Consequently, $\bq \in \bs'_1$, which yields $\bs_1 \subseteq \bs'_1$.
Thus $\bs_1 = \bs'_1$ and so $\varphi(\bs_1) = \varphi(\bs'_1)$, a contradiction.

\textbf{Case 2.} $c(L_1(\bs_1))\cap c(L_2(\bs_1)) \neq \emptyset$.

We first prove that $\varphi(\bs'_1)\subseteq\varphi(\bs_1)$. Indeed,
let $\bp$ be an arbitrary word in $\bs'_1$.
Then the inequality $\bp \preceq \bs_1$ holds in $\mathcal{V}$, and is also satisfied by $S_7$ and $S_{53}$.
By Lemma~\ref{lem5301}, for any $\bw \in S_2(\bp)$,
there exists $\bt \in L_2(\bs_1)$ such that $c(\bt) \subseteq c(\bw)$, and so $\bw=\bt$.
This implies that every subword of $\mathbf{p}$ of length $2$ is linear,
and hence $\mathbf{p}$ itself is linear.

We now show that $\ell(\bp) \leq 2$.
Indeed, suppose that $\ell(\bp) \geq 3$. We may write $\bp=y_1y_2\cdots y_m$ with $m\geq3$.
By Lemma~\ref{lem5301},
the three words $y_1y_2$, $y_2y_3$, and $y_3y_1$ all belong to $\bs_1$, contradicting condition $(c)$.
Thus $\ell(\bp) \leq 2$.

If $\ell(\bp)=2$, then by Claim~\ref{claim01} $\bp \in \bs_1$, and so $\varphi(\bp) \in \varphi(\bs_1)$.
Now suppose that $\ell(\bp)=1$. Then $\bp$ itself is a variable, and
Lemma~\ref{lem5301} implies that $\bp\in c(\bs_1)$.
Recall that $\bs_1$ is the disjoint union of $L_1(\bs_1)$ and $L_2(\bs_1)$.
If $\bp$ occurs in $L_1(\bs_1)$, then $\bp \in \bs_1$, and so $\varphi(\bp) \in \varphi(\bs_1)$;
otherwise, there exists $\bt \in L_2(\bs_1)$ such that $\bt=\bp\bp'$ for some $\bp'\in X$.

From condition (c), the graph associated with $L_2(\mathbf{s}_1)$ contains no odd cycle, and is therefore bipartite.
By Claim~\ref{claim02},
no two distinct vertices in $c(L_1(\mathbf{s}_1)) \cap c(L_2(\mathbf{s}_1))$ are connected by an odd-length path in this graph.
Applying Lemma~\ref{graph01}, there exists a bipartition $(Y, Z)$ of the graph $L_2(\mathbf{s}_1)$ such that
$c(L_1(\mathbf{s}_1)) \cap c(L_2(\mathbf{s}_1)) \subseteq Y$.

If $\bp\in Z$, let $\psi: X \to S_7$ be a substitution
defined by $\psi(t)=1$ if $t\in Z$, and $\psi(t)=a$ otherwise.
It is easy to see that $\psi(\bp)=1$ and $\psi(\bs_1)=a$,
and hence $\bp\preceq \bs_1$ does not hold in $S_7$, a contradiction.
Therefore, $\bp\in Y$.

Let $\bp^*$ denote the connected component of $\bp$ in the graph associated with $L_2(\bs_1)$.
We shall show that there exists $x\in c(L_1(\bs_1)) \cap c(L_2(\bs_1))$ such that $x\in \bp^*$.
Suppose that it is not true.
Consider the substitution $\alpha: X \to S_7$
defined by
\[
\alpha(t) =
\begin{cases}
1 & \text{if }  t \in Y \cap \mathbf{p}^* \text{ or } t \in Z \setminus \mathbf{p}^*, \\
a & \text{otherwise}.
\end{cases}
\]
A detailed verification yields $\alpha(\bp)= 1$ and $\alpha(\bs_1)=a$, a contradiction.
Hence such an $x$ exists.
So we may assume that $\bs_1$ contains $x, xx_1, x_1x_2,\dots, x_{2k+1}\bp$ for some $x_1, x_2,\dots, x_{2k+1}\in X$.
By \eqref{26042302}, $\varphi(x)=y_1$, and so $\varphi(\bp)=y_1$.
Consequently, $\varphi(\bp)=\varphi(x)$.
Hence $\varphi(\bp) \in \varphi(\bs_1)$ and so $\varphi(\bs'_1)\subseteq\varphi(\bs_1)$.

We now use $\varphi(\bs'_1)\subseteq \varphi(\bs_1)$ to prove the reverse inclusion $\varphi(\bs_1)\subseteq\varphi(\bs'_1)$.
Let $\bq$ be an arbitrary word in $\varphi(\bs_1)$.
Then $\bq \preceq \varphi(\bs'_1)$ holds in $\mathcal{V}$, and hence is satisfied by $S_2$, $S_7$ and $S_{53}$.
By \eqref{26042302}, $\varphi(\bs_1)\subseteq \bu^{(n)}$. Since $\varphi(\bs'_1)\subseteq\varphi(\bs_1)$,
it follows immediately that $\varphi(\bs'_1) \subseteq \bu^{(n)}$.
This implies that $\ell(\bw)\leq2$ and $\bw$ is linear for all $\bw\in \varphi(\bs'_1)$.
One can also show that $L_2(\varphi(\mathbf{s}_1)) \subseteq \varphi(\mathbf{s}'_1)$;
the proof is similar to that of Claim~\ref{claim01}.

If the graph associated with $L_2(\varphi(\bs'_1))$ does not contain any odd cycle,
then the remainder of the argument follows the same pattern as in Case~1 or Case~2,
depending on whether the intersection $c(L_1(\varphi(\mathbf{s}_1'))) \cap c(L_2(\varphi(\mathbf{s}'_1)))$ is empty or not;
we omit the details.

Now suppose that the graph associated with $L_2(\varphi(\bs'_1))$ contains an odd cycle.
Since $\varphi(\bs'_1) \subseteq \bu^{(n)}$,
we deduce that $\varphi(\bs'_1)$ contains
$x_1x_2+x_2x_3+\cdots+x_{2n+1}x_1$.
Consequently, $\varphi(\bs'_1)$ is either equal to $\mathbf{u}^{(n)}$ or one of the following forms:
\[
x_1x_2 + x_2x_3 + \cdots + x_{2n+1}x_1,
\]
\[
x_1x_2 + x_2x_3 + \cdots + x_{2n+1}x_1 + y_1,
\]
\[
x_1x_2 + x_2x_3 + \cdots + x_{2n+1}x_1 + y_1y_2.
\]

If $\varphi(\bs'_1)=\bu^{(n)}$, then $\varphi(\bs_1)\subseteq\varphi(\bs'_1)$ follows immediately;
otherwise, it is easy to see that $L_1(\varphi(\bs'_1))\cap c(L_2(\varphi(\bs'_1)))=\emptyset$.
Since $\bq \preceq \varphi(\bs'_1)$ holds in $S_2$,
we have by Lemma~\ref{lem201} that $\ell(\bq)\leq 2$.
If $\ell(\bq)=1$, then $\bq \preceq \varphi(\bs'_1)$ is trivial, and hence $\bq\in\varphi(\bs'_1)$.
If $\ell(\bq)=2$, then $\bq\in\varphi(\bs'_1)$, since $L_2(\varphi(\mathbf{s}_1)) \subseteq \varphi(\mathbf{s}'_1)$.
We now conclude that $\varphi(\bs_1)\subseteq\varphi(\bs'_1)$.
Thus $\varphi(\bs_1)=\varphi(\bs'_1)$, a contradiction.
This completes the proof.
\end{proof}

In the remainder we apply Theorem \ref{thm1} to prove that several ai-semirings are nonfinitely based.

\begin{cor}\label{cor124}
The ai-semiring $S_{(4,124)}$ is nonfinitely based.
\end{cor}
\begin{proof}
It is straightforward to verify that $S_{(4,124)}$ contains a copy of $S_2$ and $S_{53}$:
the subset $\{1,2,4\}$ is isomorphic to $S_2$, and $\{1,2,3\}$ is isomorphic to $S_{53}$.
Hence $\mathsf{V}(S_{(4,124)})$ contains both $S_2$ and $S_{53}$.
Moreover, $S_7$ is isomorphic to the quotient algebra $S_{(4, 124)}/\rho$,
where $\{1, 2\}$ is the only nontrivial block of $\rho$.
Consequently, $\mathsf{V}(S_{(4, 124)})$ also contains $S_7$.
Thus, it suffices to show that $\mathsf{V}(S_{(4, 124)})$ is a subvariety of $\mathcal{W}$,
i.e., that $S_{(4, 124)}$ satisfies the inequalities $\bq^{(n)} \preceq \bu^{(n)}$ for all $n\geq1$.

Observe that $\varphi(\mathbf{u}^{(n)})$ cannot be $4$.
Indeed, if $\varphi(\mathbf{u}^{(n)}) = 4$,
then
\[
\varphi(x_1)\varphi(x_2)=\varphi(x_2)\varphi(x_3)=\cdots=\varphi(x_{2n})\varphi(x_{2n+1})=\varphi(x_{2n+1})\varphi(x_1)=4.
\]
This implies that $\varphi(x_1)=3=4$, a contradiction.
Hence $\varphi(\mathbf{u}^{(n)}) \in \{1, 2, 3\}$.

If $\varphi(\bu^{(n)})=1$, then $\varphi(\bq^{(n)}) \leq \varphi(\bu^{(n)})$,
since $1$ is the additive top element of $S_{(4, 124)}$.
If $\varphi(\bu^{(n)})=2$, then either $\varphi(y_2)=2$ or $\varphi(y_2)=3$, and so $\varphi(\bq^{(n)}) \leq \varphi(\bu^{(n)})$.
If $\varphi(\bu^{(n)})=3$, then $\varphi(x_i)=\varphi(y_j)=3$ for all $1\leq i\leq 2n+1$ and $1\leq j\leq 2$,
which again yields $\varphi(\bq^{(n)}) \leq \varphi(\bu^{(n)})$.
Thus $S_{(4, 124)}$ satisfies $\bq^{(n)} \preceq \bu^{(n)}$,
and consequently $\mathsf{V}(S_{(4, 124)})$ is a subvariety of $\mathcal{W}$.

Therefore, Theorem~\ref{thm1} implies that $S_{(4, 124)}$ is nonfinitely based.
\end{proof}

Next, we prove that the 6-element ai-semiring $R_6$,
whose Cayley tables are given in Table~\ref{6-element ai-semirings}, is nonfinitely based.
We first establish the following result.

\begin{pro}\label{proR6}
$\mathsf{V}(R_6)=\mathsf{V}(S_2, S_7, S_{53})$,
where the latter denotes the variety generated by $S_2$, $S_7$, and $S_{53}$.
\end{pro}
\begin{proof}
Let $S_{(4, 359)}$ denote the $4$-element ai-semiring with Cayley tables shown in Table~\ref{tb359}.
It is straightforward to check that $R_6$ is isomorphic to a subdirect product of $S_2$ and $S_{(4, 359)}$
via the congruences defined by the nontrivial blocks $\{1, 2, 3, 4\}$ and $\{\{1, 6\}, \{2, 5\}\}$,
and that $S_{(4, 359)}$ is isomorphic to a subdirect product of $S_7$ and $S_{53}$
via the congruences defined by the nontrivial blocks $\{1, 2\}$ and $\{1, 4\}$.
Therefore,
\[
\mathsf{V}(R_6)=\mathsf{V}(S_2, S_7, S_{53}).
\]
This completes the proof.
\end{proof}

\begin{table}[htbp]
\caption{The Cayley tables of $R_6$}\label{6-element ai-semirings}
\begin{tabular}{c|cccccc}
$+$               &1 & 2 & 3 & 4 & 5 & 6\\
\hline
                 1& 1 & 2 & 1 & 1 & 2 & 1 \\
                 2&2 & 2 & 2 &  2 & 2 & 2\\
                 3&1 & 2 & 3 & 1 & 2 & 1\\
                 4&1 & 2 & 1 & 4 & 2 & 1\\
                 5&2 & 2 & 2 & 2 & 5 & 2\\
                 6&1 & 2 & 1 & 1 & 2 & 6
\end{tabular}
\qquad\qquad
\begin{tabular}{c|cccccc}
$\cdot$               &1 & 2 & 3 & 4 & 5 & 6\\
\hline
                   1&2 & 2 & 1 & 2 & 2 & 2 \\
                   2&2 & 2 & 2 & 2 & 2 & 2\\
                   3&1 & 2 & 3 & 4 & 2 & 1\\
                   4&2 & 2 & 4 & 2 & 2 & 2\\
                   5&2 & 2 & 2 & 2 & 2 & 2\\
                   6&2 & 2 & 1 & 2 & 2 & 5
\end{tabular}
\end{table}

\begin{table}[ht]
\caption{The Cayley tables of $S_{(4, 359)}$} \label{tb359}
\begin{tabular}{c|cccc}
$+$      &$1$&$2$&$3$&$4$\\
\hline
$1$      &$1$&$2$&$1$&$1$\\
$2$      &$2$&$2$&$2$&$2$\\
$3$      &$1$&$2$&$3$&$1$\\
$4$      &$1$&$2$&$1$&$4$\\
\end{tabular}\qquad
\begin{tabular}{c|cccc}
$\cdot$      &$1$&$2$&$3$&$4$\\
\hline
$1$      &$2$&$2$&$1$&$2$\\
$2$      &$2$&$2$&$2$&$2$\\
$3$      &$1$&$2$&$3$&$4$\\
$4$      &$2$&$2$&$4$&$2$\\
\end{tabular}
\end{table}

\begin{cor}
The ai-semiring $R_6$ is nonfinitely based.
\end{cor}
\begin{proof}
This follows from Theorem~\ref{thm1} and Proposition~\ref{proR6} immediately.
\end{proof}

Finally, one can use the approach in the proof of Corollary~\ref{cor124}
to show that for any $i$ in the set
\[
\{1, 2, \dots, 12, 14, 16, 20, 23, 26, 27, \dots,31, 37, 38, 53, 54, 56, 57, \dots,61\},
\]
the $3$-element ai-semiring $S_i$ (whose Cayley tables can be found in \cite{zrc})
satisfies $\bq^{(n)}\preceq \bu^{(n)}$ for all $n\geq1$.
Consequently, every $S_i$ lies in $\mathcal{W}$.
Let $\mathcal{W}_1$ denote the variety generated by all such $S_i$.
Then $\mathcal{W}_1$ is a subvariety of $\mathcal{W}$, and contains $S_2$, $S_7$, and $S_{53}$.
By Theorem~\ref{thm1}, we obtain the following.

\begin{cor}
Let $\mathcal{V}$ be an ai-semiring variety.
If $\mathcal{V}$ is a subvariety of $\mathcal{W}_1$ and contains $S_2$, $S_7$, and $S_{53}$,
then $\mathcal{V}$ is nonfinitely based.
In particular, if $\mathcal{K}$ is a class consisting of some of the above $S_i$ and contains $S_2$, $S_7$, and $S_{53}$, then the variety generated by $\mathcal{K}$ is nonfinitely based.
\end{cor}

\section{Conclusion}
We have shown that the four-element ai-semiring $S_{(4,124)}$ is nonfinitely based,
and that the variety $\mathsf{V}(S_{(4,124)})$ contains $S_7$, thereby adding a new piece of evidence for Problem~\ref{prob0127}.
Consequently, among the $866$ four-element ai-semirings, only three remain whose finite basis problem is still unresolved.

Our proof relies on a new sufficient condition developed specifically for this algebra, as all previously known criteria fail for $S_{(4,124)}$.
Whether a general criterion exists that would settle Problem~\ref{prob0127} for all finite ai-semirings containing $S_7$ remains an open question.

We would like to emphasize that, combining our result with the results obtained in \cite{jrz, gmrz, rjzl, rlyc, rlzc, ryy, wrz, yrg, yrzs}, Problem~\ref{prob0127} is now confirmed to hold for all four-element ai-semirings; that is, every four-element ai-semiring whose variety contains $S_7$ is nonfinitely based.

The original motivation of this work also included an attempt to settle the finite basis problem for the variety generated by all three-element ai-semirings.
Although this goal has not been fully achieved, we are able to show that the variety generated by $32$ explicitly listed three-element ai-semirings is nonfinitely based.
The specific form of $\mathbf{u}^{(n)}$ was designed precisely to handle this larger class.
If one were only aiming to prove the nonfinite basis property for $S_{(4,124)}$, the two terms $y_1y_2$ and $y_2y_1$ could be simplified to just one of them, say $y_1y_2$ alone.
The presence of both terms is essential, however, for the verification that the $32$ semirings satisfy $\mathbf{q}^{(n)} \preceq \mathbf{u}^{(n)}$.
Thus, the current formulation of $\mathbf{u}^{(n)}$ is a deliberate technical choice to accommodate a broader family of examples,
even though the ultimate question concerning the full three-element ai-semiring variety remains open.

Consequently, while a full resolution of the finite basis problem for the variety generated by all three-element ai-semirings remains open, our results provide a robust partial answer and introduce technical tools that may prove useful in its ultimate settlement. Whether the variety generated by all three-element ai-semirings is indeed nonfinitely based, and whether the methods developed here can be extended to settle this question, are problems we leave for future investigation.

\qquad

\noindent
\textbf{Acknowledgements}
The authors would like to thank
Professor Marcel Jackson for valuable discussions on earlier attempts to resolve
the finite basis problem for $S_{(4,124)}$
and for sharing his insights into the difficulties involved.
We also thank Zidong Gao and Yanan Wu for helpful discussions and contributions to this work.

\qquad

\noindent
\textbf{Statements and Declarations}

\quad

\noindent
\textbf{Competing Interests}: The authors declare that they have no competing financial or non-financial interests that are directly or indirectly related to the work submitted for publication.

\quad

\noindent
\textbf{Funding}: Miaomiao Ren is supported by the National Natural Science Foundation of China (Grant Nos. 12371024, 12571020). 
Mengya Yue declares no funding was received for this research.

\bibliographystyle{amsplain}


\end{document}